\documentclass{amsart}
\usepackage{amsmath, amssymb, amsthm, amsfonts}
\usepackage{hyperref}
\usepackage{cleveref}
\usepackage[a4paper,left=20mm,right=20mm,top=30mm,bottom=25mm]{geometry}
\usepackage{blkarray}
\usepackage{enumerate}
\usepackage{xcolor}

\usepackage{graphicx}
\usepackage[all]{xy} 
\usepackage{mathrsfs}
\usepackage{eufrak}
\usepackage{comment}

\theoremstyle{plain}
\newtheorem{thm}{Theorem}[section]
\newtheorem{lem}[thm]{Lemma}
\newtheorem{prop}[thm]{Proposition}

\theoremstyle{definition}

\theoremstyle{remark}
\newtheorem{rem}[thm]{Remark}

\newcommand{\res}{\textrm{res}}

\newcommand{\bl}{\textrm{bl}}
\newcommand{\pr}{\textrm{pr}}

\begin{document}

\title[{Effective results on projective normality of the first and second secant varieties}]{Effective results on projective normality of\\ the first and second secant varieties}

\author{Doyoung Choi}
\address{School of Mathematics, Korea Institute for Advanced Study, 85 Hoegiro, Dongdaemun-gu, Seoul 02455, Republic of Korea}
\email{cdy4019@kias.re.kr}

\author{Jinhyung Park}
\address{Department of Mathematical Sciences, KAIST, 291 Daehak-ro, Yuseong-gu, Daejeon 34141, Republic of Korea}
\email{parkjh13@kaist.ac.kr}

\date{\today}

\thanks{D.C. was supported by a KIAS individual grant from Korea Institute for Advanced Study (MG105101). J.P. was partially supported by the National Research Foundation (NRF) funded by the Korea government (MSIT) (RS-2026-25478877).}

\begin{abstract}
In joint work with Lacini and Sheridan, we proved that the first and second secant varieties of a smooth projective complex variety embedded by the complete linear system of a sufficiently positive line bundle are projectively normal. The purpose of this paper is to establish effective results on how positive the embedding line bundle must be for this result to hold. We also provide effective conditions under which the defining ideal of the first secant variety is generated by cubics, and furthermore, generated by $3 \times 3$-minors of a matrix of linear forms. The latter result gives an effective version of a theorem of Agostini and the second author. 
\end{abstract}

\maketitle

\section{Introduction}

\noindent Let $X$ be a smooth projective complex variety of dimension $n$, and $L$ be a very ample line bundle on $X$ giving an embedding
$$
X \subseteq \mathbb{P} H^0(X, L) = \mathbb{P}^r.
$$
The \emph{$k$-th secant variety} of $X \subseteq \mathbb{P}^r$ is
$$
\Sigma_k = \Sigma_k(X, L) := \overline{\bigcup_{x_0, \ldots, x_{k} \in X} \langle x_0, \ldots, x_{k} \rangle} \subseteq \mathbb{P}^r, 
$$
which is an irreducible projective variety. 
Note that $X=\Sigma_0 \subseteq \Sigma_1 \subseteq \Sigma_2 \subseteq \cdots \subseteq \mathbb{P}^r$.
It is a natural problem to study what kinds of singularities secant varieties possess. After the pioneering work on secant varieties of algebraic curves by Bertram \cite{Bertram} and Vermeire \cite{Vermeire}, there has been substantial work showing that secant varieties have mild singularities when the embedding line bundle is sufficiently positive (see e.g., \cite{Choi.Lacini.Park.Sheridan.25, Chou.Song.18, Ein.Niu.Park.20, Olano.Raychaudhury.Song, Ullery.16}). It was proven that the $k$-th secant variety $\Sigma_k(X, L)$ has normal Du Bois singularities when $k=1$ by Chou--Song \cite[Theorem 1.2]{Chou.Song.18} and Ullery \cite[Corollary C]{Ullery.16}, $n=1$ by Ein--Niu--Park \cite[Theorem 1.1]{Ein.Niu.Park.20}, and $k=2$ or $n=2$ by Choi--Lacini--Sheridan--Park \cite[Theorem A]{Choi.Lacini.Park.Sheridan.25} under the assumption that $L$ is sufficiently positive. Moreover, in all these cases, the secant varieties are projectively normal. 

\medskip

The primary aim of this paper is to make these results effective for $k=1$ and $k=2$. For this purpose, we consider a line bundle on $X$ of the form 
$$
L = \omega_X \otimes H^m \otimes B
$$
with $H$ is very ample and $B$ is nef for an integer $m \geq 0$. Note that $L$ becomes $p$-very ample as soon as $m \geq n+1+p$. Ein--Lazarsfeld \cite[Theorem 1]{Ein.Lazarsfeld.93} showed that if $m \geq n+1$, then $X \subseteq \mathbb{P}^r$ is projectively normal. 
If $L$ is $(2k+1)$-very ample, then $\operatorname{Sing}(\Sigma_k) = \Sigma_{k-1}$
except when $(X, L) = (\mathbb{P}^1, \mathcal{O}_{\mathbb{P}^1}(2k+1))$, in which case, $\Sigma_k = \mathbb{P}^{2k+1}$ \cite[Corollary E]{Choi.Lacini.Park.Sheridan.25}. Since a sharp effective result has already been established for the case of algebraic curves \cite{Ein.Niu.Park.20}, we henceforth assume that $n \geq 2$. 
It is known that $\Sigma_1$ has normal Du Bois singularities when $m \geq 2n+2$ (see \cite[Theorem 1.2]{Chou.Song.18} and \cite[Corollary C]{Ullery.16}) and $\Sigma_1 \subseteq \mathbb{P}^r$ is projectively normal when $m \geq 4n$ (see \cite[Theorem G]{Choi.Lacini.Park.Sheridan.25}). The first main result of this paper is the following theorem, which, in particular, shows that $\Sigma_1 \subseteq \mathbb{P}^r$ is projectively normal as soon as $m \geq 2n+2$. The key point is projective normality, from which the remaining assertions easily follow.

\begin{thm}\label{thm:main1}
Assume that $n \geq 2$. For an integer $m \geq 0$, let $L := \omega_X \otimes H^m \otimes B$ be a line bundle on $X$ with $H$ very ample and $B$ nef. When either $k=1$ and $m \geq 2n+2$ or $k=2$ and $m \geq 3n+3$, we have the following:
\begin{enumerate}
   \item $\Sigma_k=\Sigma_k(X, L)$ has normal Du Bois singularities.
   \item $\Sigma_k \subseteq \mathbb{P} H^0(X, L) = \mathbb{P}^r$ is projectively normal.
   \item $H^i(\Sigma_k, \mathcal{O}_{\Sigma_k}(\ell))=0$ for $i > 0$ and $\ell > 0$. 
   \item $\Sigma_k$ is Cohen--Macaulay if and only if $H^i(X, \mathcal{O}_X)=0$ for $1 \leq i \leq n-1$. In this case, $\Sigma_k \subseteq \mathbb{P}^r$ is arithmetically Cohen--Macaulay.
   \item $\Sigma_k$ has rational singularities if and only if $H^i(X, \mathcal{O}_X)=0$ for $1 \leq i \leq n$. 
\end{enumerate}
\end{thm}

There has been a great deal of work on equations defining secant varieties (see e.g., \cite{Agostini.Park.25, BBF24, BGL13, Choi.Lacini.Park.Sheridan.25, Ein.Niu.Park.20, Raicu, SV2011}). It is well-known that the $k$-th secant variety $\Sigma_k$ is not contained in hypersurfaces of degree $k+1$ in $\mathbb{P}^r$. The best one can hope is that the homogeneous ideal $I(\Sigma_k(X, L))$ is generated in degree $k+2$. This is the case when $n \leq 2$ or $k\leq 2$ and $L$ is sufficiently positive (see \cite[Theorem 1.2]{Ein.Niu.Park.20} for the curve case and \cite[Theorem B]{Choi.Lacini.Park.Sheridan.25} for higher dimensional cases). Moreover, Ein--Lazarsfeld \cite[Theorem 1]{Ein.Lazarsfeld.93} showed that if $L = \omega_X \otimes H^m \otimes B$ with $H$ very ample and $B$ nef and $m \geq n+2$, then the homogeneous ideal $I(X, L)$ is generated by quadrics. The second main result of this paper provides an effective condition on $m$ under which the homogeneous ideal $I(\Sigma_1(X, L))$ of the first secant variety $\Sigma_1 \subseteq \mathbb{P}^r$ is generated by cubics. A sharp effective result has been given in \cite{Ein.Niu.Park.20} for $n=1$, so we assume that $n \geq 2$. 

\begin{thm}\label{thm:main2}
Assume that $n \geq 2$. For an integer $m \geq 0$, let $L:=\omega_X \otimes H^m \otimes B$ be a line bundle on $X$ with $H$ very ample and $B$ nef. If $m \geq 3n+2$, then the homogeneous ideal $I(\Sigma_1(X, L))$ is generated by cubics.
\end{thm}

Eisenbud--Koh--Stillman \cite{EKS1988} and Sidman--Smith \cite{SS2011} proved that the homogeneous ideal $I(X, L)$ of $X \subseteq \mathbb{P}^r$ is generated by $2 \times 2$-minors of a matrix of linear forms when $L$ is sufficiently positive. 
More precisely, if $L = L_1 \otimes L_2$ such that $L_1:=\omega_X \otimes H^{m_1} \otimes B_1$ and $L_2:=\omega_X \otimes H^{m_2} \otimes B_2$ with $H$ very ample and $B_1, B_2$ nef for integers $m_1, m_2 \geq n+2$, then $I(X, L)$ is generated by $2 \times 2$-minors of  the \emph{catalecticant matrix}
$$
\operatorname{Cat}(L_1, L_2) := \begin{pmatrix} 
\alpha_1 \beta_1 & \alpha_1 \beta_2 & \cdots & \alpha_1 \beta_b \\
\alpha_2 \beta_1 & \alpha_2 \beta_2 & \cdots & \alpha_2 \beta_b \\
\vdots & \vdots & \ddots & \vdots \\
\alpha_a \beta_1 & \alpha_a \beta_2 & \cdots & \alpha_a \beta_b \\
\end{pmatrix},
$$
where $\alpha_1, \ldots, \alpha_a$ is a basis of $H^0(X, L_1)$, and $\beta_1, \ldots, \beta_b$ is a basis of $H^0(X, L_2)$ (see \cite[Theroem 1.3]{SS2011}). Although the catalecticant matrix $\operatorname{Cat}(L_1, L_2)$ depends on the choice of bases of $H^0(X, L_1)$ and $H^0(X, L_2)$, it is elementary to see that the ideal generated by $(k+2) \times (k+2)$-minors of  $\operatorname{Cat}(L_1, L_2)$ is independent of the choice of bases. 
It was conjectured that the homogeneous ideal $I(\Sigma_k(X, L))$ of the $k$-th secant variety $\Sigma_k \subseteq \mathbb{P}^r$ is generated by $(k+2) \times (k+2)$-minors of $\operatorname{Cat}(L_1, L_2)$ when $L=L_1 \otimes L_2$ and $L_1, L_2$ are sufficiently positive (see \cite[Remark in page 518]{EKS1988} and \cite[Conjecture 1.2]{SS2011}). This conjecture was verified when $n \leq 2$ or $k \leq 1$ by Agostini--Park \cite[Theorem B]{Agostini.Park.25}. It is worth noting that the conjecture is known to be false when $X$ is a singular curve \cite{BGL13} or $n$ and $k$ are sufficiently large \cite{BBF24} (see also the discussion right after \cite[Conjecture 1.2]{SS2011}). The final main result of this paper is to make \cite[Theorem B]{Agostini.Park.25} for $k=1$ effective. A fairly nice effective result has been already obtained for $n=1$ \cite[Theorem A]{Agostini.Park.25}, so we assume that $n \geq 2$. An effective result for $\Sigma_1$ to be determinantally presented had been known in the case of Veronese embeddings \cite{Raicu}.

\begin{thm}\label{thm:main3}
Assume that $n \geq 2$. Let $L_1:=\omega_X \otimes H^{m_1} \otimes B_1$ and $L_2:=\omega_X \otimes H^{m_2} \otimes B_2$ with $H$ very ample and $B_1, B_2$ nef for integers $m_1, m_2 \geq 0$, and $L:=L_1 \otimes L_2$. 
If $m_1, m_2 \geq 3n+2$, then the homogeneous ideal $I(\Sigma_1(X, L))$ is generated by $3 \times 3$-minors of $\operatorname{Cat}(L_1, L_2)$. 
\end{thm}

The proofs of the main theorems are carried out following the cohomological approach developed in \cite{Agostini.Park.25, Choi.Lacini.Park.Sheridan.25, Ein.Niu.Park.20}. All the theorems are ultimately reduced to cohomology vanishing for vector bundles on Hilbert schemes of points. This can be derived from cohomology vanishing for line bundles on nested Hilbert schemes of points, to which Kawamata--Viehweg vanishing theorem is applied. Since we are finally working with $X^{[1,2,3]}$ and $X^{[1,2,3,4]}$, the fact that they have Gorenstein canonical singularities is a key ingredient of this paper. This was proved in \cite[Lemma 4.20]{Choi.Lacini.Park.Sheridan.25} and \cite[Corollary 4.3]{Choi.25}, respectively.

\medskip

This paper is organized as follows. We begin in Section \ref{sec:prelim} with reviewing Hilbert schemes of points and secant varieties. In Section \ref{sec:mainvanishing}, we establish cohomology vanishing statements on nested Hilbert schemes of points, which play the crucial role in proving the main theorems. The proofs of the main theorems are given in Section \ref{sec:proofs}.

\section{Preliminaries}\label{sec:prelim}
\noindent In this section, we collect necessary facts on Hilbert schemes of points and secant varieties. We closely follow the notations in \cite{Choi.Lacini.Park.Sheridan.25} and \cite{Choi.Park.26}, and we refer to those papers and the references therein for further details. Throughout the paper, let $X$ be a smooth projective complex variety of dimension $n$, and $L$ be a line bundle on $X$.

\subsection{Hilbert schemes of points}
The \emph{Hilbert scheme} of $k$ points on $X$ is set-theoretically
$$
X^{[k]}:= \{ \xi \subseteq X \mid \dim \xi=0, \operatorname{length}(\xi)=k\}.
$$
In this paper, we only deal with the case where $k \leq 4$. In this case, $X^{[k]}$ is an irreducible projective variety of dimension $kn$, and furthermore, $X^{[k]}$ is smooth if and only if $n \leq 2$ or $k \leq 3$ (see \cite{Cheah}). There is the \emph{universal family}
$$
\mathcal{Z}_{k}:=\{ (x, \xi) \in X \times X^{[k]} \mid x \in \xi\} \subseteq X \times X^{[k]}
$$
over $X^{[k]}$ with projection maps $\pr_1 \colon \mathcal{Z}_{k} \to X$ and $\pr_2 \colon \mathcal{Z}_{k} \to X^{[k]}$. If $n \leq 2$ or $k \leq 3$, then $\mathcal{Z}_k$ has rational singularities (see \cite[\S 1.2.4]{Choi.Lacini.Park.Sheridan.25}). Let 
$$
E_{k,L}:=\pr_{2,*} \pr_1^* L
$$
be the \emph{tautological bundle} associated to $L$, which is a rank $k$ vector bundle on $X^{[k]}$ since $\pr_2$ is a finite flat morphism of degree $k$. Notice that $H^0(X^{[k]}, E_{k,L}) = H^0(X, L)$ and the fiber of $E_{k,L}$ over $\xi \in X^{[k]}$ can be identified with $H^0(\xi, L|_{\xi})$. In particular, $E_{k,L}$ is globally generated when $L$ is $(k-1)$-very ample, i.e., $H^0(X, L) \to H^0(\xi, L|_{\xi})$ is surjective for every $\xi \in X^{[k]}$. In this case, the evaluation map $H^0(X, L) \otimes \mathcal{O}_{X^{[k]}} \to E_{k,L}$ is surjective, so its kernel $M_{k,L}$ is a vector bundle on $X^{[k]}$. We have a short exact sequence
$$
0 \longrightarrow M_{k,L} \longrightarrow H^0(X, L) \otimes \mathcal{O}_{X^{[k]}} \longrightarrow E_{k,L} \longrightarrow 0.
$$
When $k=1$, we simply write $M_L:=M_{1,L}$. If $\pr_2 \colon X \times X^{[k]} \to X^{[k]}$ is the second projection, then $M_{k,L} = \operatorname{pr}_{2,*}\mathscr{I}_{\mathcal{Z}_k} \otimes (L \boxtimes \mathcal{O}_{X^{[k]}})$. 
We now make a few comments on which line bundles on $X$ are $(k-1)$-very ample. It is well-known that the line bundle $\omega_X \otimes H^m \otimes B$ on $X$ with $H$ very ample and $B$ nef is globally generated as soon as $m \geq n+1$. By \cite[Theorem 1.1]{HTT}, if $m \geq n+k$, then $\omega_X \otimes H^m \otimes B$ is $(k-1)$-very ample. We will frequently use this fact in Sections \ref{sec:mainvanishing} and \ref{sec:proofs}.

\medskip

Now, we consider the Hilbert--Chow morphism $h_k \colon X^{[k]} \to X^{(k)}$. The $k$-th symmetric product $X^{(k)}$ is obtained by the quotient of the ordinary $k$-th product $X^k$ by the symmetric group action $\mathfrak{S}_k$ permuting the components. The line bundle $L^{\boxtimes k}$ on $X^k$ descends to a line bundle $S_{k,L}$ on $X^{(k)}$. We define the line bundles
$$
T_{k,L}:= h_k^* S_{k,L}, ~N_{k,L}:= T_{k,L}(-\delta_k),~\text{and}~ A_{k,L}:= T_{k,L}(-2\delta_k) = N_{k,L}(-\delta_k)
$$
on $X^{[k]}$, where $\delta_k$ is a divisor on $X^{[k]}$ with $\mathcal{O}_{X^{[k]}}(-\delta_k) = \det E_{k, \mathcal{O}_X}$.  If $k=1$, then $\delta_1=0$ and $E_{1,L}=T_{1,L}=N_{1,L}=A_{1,L}=L$.  Note that $N_{k,L}=\det E_{k,L}$. Thus $N_{k,L}$ is globally generated when $L$ is $(k-1)$-very ample. By \cite[Theorem 1.1]{HTT}, $H^{k-1}$ is $(k-1)$-very ample for a very ample line bundle $H$ on $X$, so $N_{k, H^{k-1}}$ is globally generated. On the other hand, if a line bundle $B$ on $X$ is nef (resp. ample), then the line bundle $S_{k,B}$ on $X^{(k)}$ is nef (resp. ample) so that the line bundle $T_{k,B}=h_k^* S_{k,B}$ is nef (resp. nef and big). 
We will frequently use these facts in Sections \ref{sec:mainvanishing} and \ref{sec:proofs}. 
We have $H^0(X^{[k]}, T_{k,L})= S^k H^0(X, L)$ and $H^0(X^{[k]}, N_{k,L})= \wedge^k H^0(X, L)$. We refer to \cite[\S 2.3.2]{Choi.Lacini.Park.Sheridan.25} for basic properties of these line bundles on $X^{[k]}$. When $k=2,3$, we can compute the canonical line bundle $\omega_{X^{[k]}} = T_{k, \omega_X}((n-2)\delta_k)$ (see \cite[\S 1.1.7]{Choi.Lacini.Park.Sheridan.25}).

\subsection{Nested Hilbert schemes of points}
For positive integers $k_1 < \cdots < k_{\ell}$, the \emph{nested Hilbert scheme} parametrizing nested sequence of zero-dimensional subschemes of length $k_1, \ldots, k_{\ell}$ on $X$ is set-theoretically
$$
X^{[k_1, \ldots, k_{\ell}]} := \{ (\xi_1, \ldots, \xi_{\ell}) \in X^{[k_1]} \times \cdots \times X^{[k_{\ell}]} \mid \xi_1 \subseteq \cdots \subseteq \xi_{\ell} \} \subseteq X^{[k_1]} \times \cdots \times X^{[k_{\ell}]}.
$$
In this paper, we only deal with the nested Hilbert schemes $X^{[k-1,k]}$ for $k=2,3,4$, $X^{[1,2,3]}$, and $X^{[1,2,3,4]}$. They are all irreducible projective varieties. Furthermore, $X^{[k-1,k]}$ is smooth for $k=2,3$ (see \cite{Cheah}), and the remaining $X^{[3,4]}$, $X^{[1,2,3]}$, and $X^{[1,2,3,4]}$  have Goresntein canonical singularities (see \cite[Theorem 1.2]{Choi.25}, \cite[Lemma 4.20]{Choi.Lacini.Park.Sheridan.25}, and \cite[Corollary 4.3]{Choi.25}, respectively).

\medskip

Assume that $k=2,3,4$, and consider the nested Hilbert scheme $X^{[k-1,k]}$. There are projection maps
$$
\rho_{k-1,k} \colon X^{[k-1,k]} \longrightarrow X^{[k]}~~\text{ and }~~\tau_{k-1,k} \colon X^{[k-1,k]} \longrightarrow X^{[k-1]}, 
$$
and there is a residual morphism
$$
\res_{k-1,k} \colon X^{[k-1,k]} \longrightarrow X,~~(\eta, \xi) \longmapsto (\mathscr{I}_{\xi}, \mathscr{I}_{\eta}). 
$$
Note that $\rho_{k-1,k}$ is generically finite and factors thorough the universal family $\mathcal{Z}_k$ over $X^{[k]}$. On the other hand, $X^{[k-1,k]}$  can be obtained by the blow-up 
$$
\bl_{k-1,k} \colon X^{[k-1, k]} \longrightarrow X \times X^{[k-1]}
$$
along the universal family $\mathcal{Z}_{k-1}$ with exceptional divisor $F_{k-1}$. Then $\res_{k-1,k} = \pr_1 \circ \bl_{k-1,k}$ and $\tau_{k-1,k} = \pr_2 \circ \bl_{k-1,k}$, where $\pr_1 \colon X \times X^{[k-1]} \to X$ and $\pr_2 \colon X \times X^{[k-1]} \to X^{[k-1]}$ are projection maps. Notice that $\mathcal{Z}_2 = X^{[1,2]}$ is obtained by the blow-up of $X \times X$ along the diagonal $\Delta \subseteq X \times X$. In particular, $\rho_2$ is finite. If $L$ is $(k-1)$-very ample, then there are short exact sequences
\begin{align*}
& 0 \longrightarrow \res_{k-1,k}^* L(-F_{k-1}) \longrightarrow \rho_{k-1,k}^* E_{k,L} \longrightarrow \tau_{k-1,k}^* E_{k-1,L} \longrightarrow 0\\
& 0 \longrightarrow \rho_{k-1,k}^* M_{k,L} \longrightarrow \tau_{k-1,k}^* M_{k-1,L} \longrightarrow \res_{k-1,k}^* L(-F_{k-1}).
\end{align*}
We have $\rho_{k-1,k}^* T_{k,L} =  \tau_{k-1,k}^* T_{k-1,L} \otimes \res_{k-1,k}^* L$ and $\rho_{k-1}^* \delta_k = \tau_{k-1}^* \delta_{k-1} + F_{k-1}$. We can also compute the canonical line bundle 
$$
\omega_{X^{[k-1,k]}} = \rho_{k-1,k}^* \omega_{X^{[k]}} (F_{k-1}) =  (\tau_{k-1,k}^* \omega_{X^{[k-1]}} \otimes \res_{k-1,k}^*\omega_X)((n-1)F_{k-1})
= \bl_{k-1,k}^* (\omega_X \boxtimes  \omega_{X^{[k-1]}})((n-1)F_{k-1}).
$$

\medskip

For the nested Hilbert scheme $X^{[1,2,3]}$, there are projection maps
$$
\rho_{1,2,3} \colon X^{[1,2,3]} \longrightarrow X^{[3]},~\tau_{1,2,3} \colon X^{[1,2,3]} \longrightarrow X^{[2]},~\text{ and }~ \rho_{1,2,3}' \colon X^{[1,2,3]} \longrightarrow X^{[2,3]},
$$
and there is a residual morphism
$$
\res_{1,2,3} \colon X^{[1,2,3]} \longrightarrow X,~~(x, \eta, \xi) \longmapsto (\mathscr{I}_{\xi}, \mathscr{I}_{\eta}). 
$$
Note that $\rho_{1,2,3}$ is generically finite and $\rho_{1,2,3}'$ is finite. We have $\rho_{1,2,3} = \rho_{2,3} \circ \rho_{1,2,3}'$. On the other hand, $X^{[1,2,3]}$ can be obtained by the blow-up
$$
\bl_{1,2,3} \colon X^{[1,2,3]} \longrightarrow X \times X^{[1,2]}
$$
along the universal family 
$$
\mathcal{W}_{1,2}:= \{ (x, (y, \xi)) \in X \times X^{[1,2]} \mid x,y \in \xi\} \subseteq X \times X^{[1,2]}
$$
over $X^{[1,2]}$ with exceptional divisor $F_{1,2}$. Then $\res_{1,2,3} = \pr_1 \circ \bl_{1,2,3}$ and $\tau_{1,2,3} = \pr_2 \circ \bl_{1,2,3}$, where $\pr_1 \colon X \times X^{[1,2]} \to X$ and $\pr_2 \colon X \times X^{[1,2]} \to X^{[1,2]}$ are projection maps. Note that $\mathcal{W}_{1,2}$ has two irreducible components. We have $\rho_{1,2,3}^* T_{k,L} = \bl_{1,2,3}^* (L \boxtimes \rho_{1,2}^* T_{2,L})$ and $\rho_{1,2,3}^* \mathcal{O}_{X^{[3]}}(\delta_3) = \bl_{1,2,3}^* (\mathcal{O}_X \boxtimes \rho_{1,2}^* \mathcal{O}_{X^{[2]}}(\delta_2)) (F_{1,2})$. We can also compute the canonical line bundle
$\omega_{X^{[1,2,3]}} = \bl_{1,2,3}^* (\omega_X \boxtimes \omega_{X^{[1,2]}}) ((n-1)F_{1,2})$.

\medskip

For the nested Hilbert scheme $X^{[1,2,3,4]}$, there is the projection map $\rho_{1,2,3,4} \colon X^{[1,2,3,4]} \to X^{[4]}$, which is generically finite. As before, $X^{[1,2,3,4]}$ can be obtained by the blow-up
$$
\bl_{1,2,3,4} \colon X^{[1,2,3,4]} \longrightarrow X \times X^{[1,2,3]}
$$
along the universal family
$$
\mathcal{W}_{1,2,3}:= \{ (x, (y,\eta, \xi)) \in X \times X^{[1,2,3]} \mid x \in \xi, y \in \eta, \eta \subseteq \xi\} \subseteq X \times X^{[1,2,3]}
$$
over $X^{[1,2,3]}$ with exceptional divisor $F_{1,2,3}$. Note that $\mathcal{W}_{1,2,3}$ has three irreducible components. We have $\rho_{1,2,3,4}^* T_{k,L} = \bl_{1,2,3,4}^* (L \boxtimes \rho_{1,2,3}^* T_{3,L})$ and $\rho_{1,2,3,4}^* \mathcal{O}_{X^{[4]}}(\delta_4) = \bl_{1,2,3,4}^* (\mathcal{O}_X \boxtimes \rho_{1,2,3}^* \mathcal{O}_{X^{[3]}}(\delta_3))(F_{1,2,3})$. We can also compute the canonical line bundle
$\omega_{X^{[1,2,3,4]}} = \bl_{1,2,3,4}^* (\omega_X \boxtimes \omega_{X^{[1,2,3]}}) ((n-1)F_{1,2,3})$.

\subsection{Secant varieties}
From now on, we assume that $k=2$ or $3$ so that $X^{[k]}$ and $X^{[k-1,k]}$ are smooth. Let 
$$
B^k=B^k(L):=\mathbb{P}(E_{k+1,L})~\text{ with canonical projection }~\pi_k \colon B^k \longrightarrow X^{[k]}.
$$
We further assume that $L$ is $(2k-1)$-very ample. Since $E_{k,L}$ is globally generated, it follows that the tautological line bundle $\mathcal{O}_{B^k}(1)$ is base point free. Note that $H^0(B^k, \mathcal{O}_{B^k}(1)) = H^0(X^{[k]}, E_{k,L})=H^0(X, L)$. Thus the complete linear system $\lvert \mathcal{O}_{B^k}(1) \rvert$ gives a morhpism $\alpha_k \colon B^k \to  \mathbb{P} H^0(X, L) = \mathbb{P}^r$
whose image is the \emph{$k$-secant variety}
$$
\sigma_k=\sigma_k(X, L):=\overline{\bigcup_{\xi \in X^{[k]}} \langle \xi \rangle} \subseteq \mathbb{P}^r,
$$
which is nothing but the $(k-1)$-th secant variety $\Sigma_{k-1}=\Sigma_{k-1}(X, L)$. Two different conventions for secant varieties are both widely used in the literature. Following \cite{Choi.Lacini.Park.Sheridan.25, Choi.Park.26}, we use the notation of $k$-secant variety $\sigma_k$ simply because it is more convenient for maintaining consistent indexing. 
Note that 
$$
\alpha_k \colon B^k \longrightarrow \sigma_k
$$
is a birational morphism hence a resolution of singularities and $\sigma_k$ is a projective variety with $\dim \sigma_k = kn+k-1$. It is clear that $\mathcal{O}_{B^k}(1) = \alpha_k^* \mathcal{O}_{\sigma_k}(1)$. In our case, $\operatorname{Sing}(\sigma_k) = \sigma_{k-1}$ by \cite[Corollary E]{Choi.Lacini.Park.Sheridan.25}.

\medskip

Let 
$$
B^{k-1,k}=B^{k-1,k}(L):=\mathbb{P}(\tau_{k-1,k}^* E_{k-1,L})~\text{ with canonical projection }~\pi_{k-1,k} \colon B^{k-1,k} \longrightarrow X^{[k-1,k]}.
$$
We have a commutative diagram
$$
\xymatrix{
B^{k-1,k} \ar[r]^-{\widetilde{\tau}_{k-1,k}} \ar[d]_-{\pi_{k-1,k}} & B^{k-1} \ar[d]^-{\pi_{k-1}} \\
X^{[k-1,k]} \ar[r]_{\tau_{k-1,k}} & X^{[k-1]}.
}
$$
On the other hand, from the surjection $\rho_{k-1,k}^* E_{k,L} \to \tau_{k-1,k}^* E_{k-1,L}$, we get a morphism
$\alpha_{k-1,k} \colon B^{k-1,k} \to B^k$ whose image is denoted by $Z_{k-1}^k$. Note that 
$$
\alpha_{k-1,k} \colon B^{k-1,k} \longrightarrow Z_{k-1}^k
$$
is a birational morphism hence a resolution of singularities and $Z_{k-1}^k=\alpha_k^{-1}(\sigma_{k-1})$ is the exceptional divisor of $\alpha_k$. By \cite[Proposition 3.29]{Choi.Lacini.Park.Sheridan.25}, 
$$
\mathcal{O}_{B^3}(-Z_{k-1}^k) =  \pi_k^* A_{k,L} \otimes \mathcal{O}_{B^k}(-k).
$$
There is a commutative diagram
$$
\xymatrix{
B^{k-1,k} \ar[r]^-{\alpha_{k-1,k}} \ar[d]_-{\pi_{k-1,k}} & Z_{k-1}^k \ar[d]^-{\pi_k} \\
X^{[k-1,k]} \ar[r]_{\rho_{k-1,k}} & X^{[k]}.
}
$$
Note that $\alpha_k \circ \alpha_{k-1,k} = \alpha_{k-1} \circ \widetilde{\tau}_{k-1,k}$ and $\alpha_k(\alpha_{k-1,k}(B^{k-1,k}))=\alpha_k(Z_{k-1}^k)=\sigma_{k-1}$. 

\medskip

When $k=2$, we have $\mathcal{Z}_2 = X^{[1,2]} = B^{1,2} = Z_1^2$, which is the universal family over $X^{[2]}$ via $\pi_2$, and the restricted morphism $\alpha_2 \colon Z_1^2 \to X$ coincides with the residual morhpism $\res_{1,2} \colon X^{[1,2]} \to X$. Notice that $Z_1^2$ is smooth. 
For $k=3$, consider the effective divisor $\widetilde{\tau}_{2,3}^* Z_1^2$ on $B^{2,3}$. Note that $\widetilde{\tau}_{2,3}^* Z_1^2 = X^{[1,2,3]}$ with 
$\pi_{2,3}|_{\widetilde{\tau}_{2,3}^* Z_1^2} = \rho_{1,2,3}'$ and $(\pi_2 \circ \widetilde{\tau}_{2,3})|_{\widetilde{\tau}_{2,3}^* Z_1^2} = \tau_{1,2,3}$. Let $Z_1^3:=\alpha_{2,3}(\widetilde{\tau}_{2,3}^*Z_1^2)$. Then $\alpha_3(Z_1^3)=X$, and $Z_1^3 = \alpha_3^{-1}(X)$. Note that $Z_1^3 = \mathcal{Z}_3$ is the universal family over $X^{[3]}$ via $\pi_3$ and the restricted morphism $\alpha_3 \colon Z_1^3 \to X$ coincides with the the projection map $\pr_1 \colon \mathcal{Z}_3 \to X$.
 Since $Z_2^3$ is a prime divisor on a smooth variety $B^3$, it is Gorenstein and $\omega_{B^{2,3}}(\widetilde{\tau}_{2,3}^* Z_1^2) = \alpha_{2,3}^* \omega_{Z_2^3}$.
We have a short exact sequence
$$
0 \longrightarrow \alpha_{2,3,*} \mathcal{O}_{B^{2,3}}(-\widetilde{\tau}_{2,3}^* Z_1^2) \longrightarrow \mathcal{O}_{Z_2^3} \longrightarrow \mathcal{O}_{Z_1^3} \longrightarrow 0,
$$
and $R^i \alpha_{2,3,*} \mathcal{O}_{B^{2,3}} = R^i \alpha_{2,3,*} \mathcal{O}_{B^{2,3}}(-\widetilde{\tau}_{2,3}^* Z_1^2) = 0$ for $i>0$ (see \cite[Subsection 2.2]{Choi.Park.26}). Here we show that $Z_2^3$ has semi-log canonical singularities. Recall that $\alpha_{2,3} \colon B^{2,3} \to Z_2^3$ is a resolution of singularities. As $K_{B^{2,3}} +  \widetilde{\tau}_{2,3}^* Z_1^2 = \alpha_{2,3}^* K_{Z_2^3}$, it suffices to prove that $(B^{2,3},  \widetilde{\tau}_{2,3}^* Z_1^2)$ is a log canonical pair. But this holds by inversion of adjunction since $\widetilde{\tau}_{2,3}^* Z_1^2 = X^{[1,2,3]}$ has canonical singularities. In particular, \cite[Corollary 6.32]{Kollar.13} shows that $Z_2^3$ has Du Bois singularities.  
We refer to \cite[Sections 2 and 3]{Choi.Lacini.Park.Sheridan.25} for more details on the geometry of $B^k$ and $B^{k-1,k}$.

\section{Cohomology Vanishing results on nested Hilbert schemes of points}\label{sec:mainvanishing}

\noindent In this section, we establish various cohomology vanishing statements, which are the main ingredients of the proofs of the main theorems. Recall that $X$ is a smooth projective complex variety of dimension $n$.

\begin{prop}\label{prop:cohvanA}
Let $k \geq 1$ and $1 \leq i \leq k$ be integers, and $L_i:=\omega_X \otimes H^{m_i} \otimes B_i$ be a line bundle on $X$ with $H$ very ample and $B_i$ nef for an integer $m_i \geq 0$. If $a \geq 1$, $m_1 \geq (k-1)(a+n-1)+1$, and $m_i \geq a+n$ for $2 \leq i \leq k$, then  
$$
H^i(X^k, \mathscr{I}_{\Delta_{1,2} \cup \cdots \cup \Delta_{1,k}}^a \otimes (L_1 \boxtimes \cdots \boxtimes L_k))=0~~\text{ for $i > 0$},
$$
where $\Delta_{1,i}:=\{ (x_1, \ldots, x_k) \in X^k \mid x_1 = x_i\}$ is a pairwise diagonal for $2 \leq i \leq k$. In particular, if $m_1 \geq (k-1)n+1$ and $m_i \geq n+1$ for $2 \leq i \leq k$, then
$$
H^i(X, M_{L_2} \otimes \cdots \otimes M_{L_k} \otimes L_1) = 0~~\text{ for $i>0$}. 
$$
\end{prop}

\begin{proof}
This is well-known, but we include the proof for the reader's convenience. 
Let $b \colon Y \to X^k$ be the blow-up of $X^k$ along $\Delta_{1,2} \cup \cdots \cup \Delta_{1,k}$ with exceptional divisors $E_{1,2}, \ldots, E_{1,k}$ such that $b(E_{1,i})=\Delta_{1,i}$ for $2 \leq i \leq k$. 
It is known (see e.g., \cite[Corollary 3.5]{BL2025}) that
$$
\mathscr{I}_{\Delta_{1,2} \cup \cdots \cup \Delta_{1,k}} 
= \mathscr{I}_{\Delta_{1,2}} \cdots \mathscr{I}_{\Delta_{1,k}} 
= \mathscr{I}_{\Delta_{1,2}} \otimes \cdots \otimes \mathscr{I}_{\Delta_{1,k}}.
$$
As $L_2, \ldots, L_k$ are globally generated and $H^i(X, L_2) = \cdots = H^i(X, L_k)=0$ for $i>0$, we find
$$
R^i \pr_{1,*}  \mathscr{I}_{\Delta_{1,2} \cup \cdots \cup \Delta_{1,k}} \otimes (L_1 \boxtimes \cdots \boxtimes L_k) = \begin{cases} M_{L_2} \otimes \cdots \otimes M_{L_k} \otimes L_1 & \text{for $i=0$} \\ 0 & \text{for $i>0$}, \end{cases}
$$
where $\pr_1 \colon X^k \to X$ is the projection given by $(x_1, \ldots, x_k) \mapsto x_1$ (see e.g., \cite[Theorem 3.6]{BL2025}). 
Thus the first assertion implies the second assertion. 
One can also easily check that $Y$ can be obtained by the successive blow-ups of along the strict transforms of $\Delta_{1,2}, \ldots, \Delta_{1,k}$ in any order. In particular, $Y$ is smooth, and for $2 \leq i \leq k$, the line bundle $D_{1,i}:=b^* \pr_{1,i}^* (H \boxtimes H) (-E_{1,i})$ on $Y$ is nef, where $\pr_{1,i} \colon X^k \to X^2$ is the projection map given by $(x_1, \ldots, x_k) \mapsto (x_1, x_i)$. Set $E:=E_{1,2} + \cdots + E_{1,k}$ and $D:=D_{1,2} \otimes \cdots \otimes D_{1,k}$. 
Note that
$$
H^i(X^k, \mathscr{I}_{\Delta_{1,2} \cup \cdots \cup \Delta_{1,k}}^a \otimes (L_1 \boxtimes \cdots \boxtimes L_k)) = H^i(Y, b^*(L_1 \boxtimes \cdots \boxtimes L_k)(-aE))~~\text{ for $i>0$}.
$$
We may write
$$
b^*(L_1 \boxtimes \cdots \boxtimes L_k)(-aE)= \omega_Y \otimes \underbrace{b^*(H^{m_1 - (k-1)(a+n-1)} \boxtimes H^{m_2 - (a+n-1)} \boxtimes \cdots \boxtimes H^{m_k-(a+n-1)})}_{\text{nef and big}} \otimes \underbrace{D^{a+n-1}}_{\text{nef}}.
$$
Thus the first assertion follows from Kawamata--Viehweg vanishing. 
\end{proof}

\begin{rem}
In Proposition \ref{prop:cohvanA}, when $a=1$ and $L_2 = \cdots = L_k$, 
a stronger result is proved in \cite[Theorem 2.1]{Ein.Lazarsfeld.93}  
\end{rem}

\begin{prop}\label{prop:cohvanB}
Let $L:=\omega_X \otimes H^m \otimes B$ be a line bundle on $X$ with $H$ very ample and $B$ nef for an integer $m \geq 0$. Then we have the following:
\begin{enumerate}
    \item Let $N$ be a nef line bundle on $X^{[1,2]}$. If $m \geq n+1$, then 
    $$
    H^i(X^{[1,2]}, \rho_{1,2}^* N_{2,L} \otimes N)=0~~\text{ for $i>0$}.
    $$
    In particular, if $m \geq n+2$, then 
    $$
     H^i(X^{[1,2]}, \rho_{1,2}^* A_{2,L} \otimes N)=0~~\text{ for $i>0$}.
    $$
    \item Let $N'$ be a nef line bundle on $X^{[1,2,3]}$, and $N$ be a nef line bundle on $X^{[2,3]}$. If $m  \geq 2n+1$, then 
    $$
    H^i(X^{[1,2,3]}, \rho_{1,2,3}^* N_{3,L} \otimes N')=0~\text{ and }~H^i(X^{[2,3]}, \rho_{2,3}^* N_{3,L} \otimes N)~~\text{ for $i>0$}.
    $$
    In particular, if $m \geq 2n+3$, then 
    $$
    H^i(X^{[1,2,3]}, \rho_{1,2,3}^* A_{3,L} \otimes N')=0~\text{ and }~H^i(X^{[2,3]}, \rho_{2,3}^* A_{3,L} \otimes N)~~\text{ for $i>0$}.
    $$
\end{enumerate}
\end{prop}

\begin{proof}
$(1)$ Since we may write
$$
\rho_{1,2}^* N_{2,L} \otimes N = \omega_{X^{[1,2]}} \otimes \underbrace{\rho_{1,2}^* T_{2,H^{m-n}}}_{\text{nef and big}}  \otimes  \underbrace{\rho_{1,2}^* (T_{2, B} \otimes N_{2,H}^n)  \otimes N}_{\text{nef}},
$$
the first assertion follows from Kawamata--Viehweg vanishing. As $\rho_{1,2}^* A_{2,L} = \rho_{1,2}^* N_{2, \omega_X \otimes H^{m-1} \otimes B} \otimes \rho_{1,2}^* N_{2, H}$ and $\rho_{1,2}^* N_{2, H}$ is nef, the first assertion implies the second assertion. 

\medskip

\noindent $(2)$ Recall that there is a finite morphism $\rho_{1,2,3}' \colon X^{[1,2,3]} \to X^{[2,3]}$ and $X^{[1,2,3]}$ has Gorenstein canonical singularities. Note that $\rho_{1,2,3} = \rho_{2,3} \circ \rho_{1,2,3}'$ and $\mathcal{O}_{X^{[2,3]}}$ is a direct summand of $\rho_{1,2,3,*}' \mathcal{O}_{X^{[1,2,3]}}$. Thus we only need to prove the cohomology vanishing on $X^{[1,2,3]}$. Since we may write
$$
\rho_{1,2,3}^* N_{3, L} \otimes N' = \omega_{X^{[1,2,3]}} \otimes \underbrace{\rho_{1,2,3}^* T_{3, H^{m-2n}}}_{\text{nef and big}} \otimes \underbrace{\rho_{1,2,3}^* (T_{3, B} \otimes N_{3, H^{2}}^n) \otimes N'}_{\text{nef}},
$$
the first assertion follows from Kawamata--Viehweg vanishing. As $\rho_{1,2,3}^* A_{3, L} = \rho_{1,2,3}^* N_{3, \omega X \otimes H^{m-2} \otimes B} \otimes \rho_{1,2,3}^* N_{3, H^2}$ and $\rho_{1,2,3}^* N_{3, H^2}$ is nef, the first assertion implies the second assertion. 
\end{proof}

\begin{prop}\label{prop:mainvanishing1}
Let $L_1:=\omega_X \otimes H^{m_1} \otimes B_1$ and $L_2:=\omega_X \otimes H^{m_2} \otimes B_2$ be line bundles on $X$ with $H$ very ample and $B_1, B_2$ nef for integers $m_1, m_2 \geq 0$. Then we have the following:
\begin{enumerate}
    \item Let $N$ be a nef line bundle $N$ on $X^{[1,2]}$. If $m_1, m_2 \geq 2n+1$, then 
    $$
    H^i(X^{[1,2]}, \rho_{1,2}^* (M_{2,L_1} \otimes N_{2,L_2}) \otimes N)=0~~\text{ for $i>0$}.
    $$
    In particular, if $m_1 \geq 2n+1$ and $m_2 \geq 2n+2$, then 
     $$
    H^i(X^{[1,2]}, \rho_{1,2}^* (M_{2,L_1} \otimes A_{2,L_2}) \otimes N)=0~~\text{ for $i>0$}.
    $$
    \item Let $N'$ be a nef line bundle on $X^{[1,2,3]}$, and $N$ be a nef line bundle on $X^{[2,3]}$. If $m_1, m_2 \geq 3n+1$, then 
    $$
    H^i(X^{[1,2,3]}, \rho_{1,2,3}^* (M_{3, L_1} \otimes N_{3,L_2}) \otimes N') = 0~\text{and}~H^i(X^{[2,3]}, \rho_{2,3}^* (M_{3, L_1} \otimes N_{3,L_2}) \otimes N)=0~\text{for $i>0$}. 
    $$
    In particular, if $m_1 \geq 3n+1$ and $m_2 \geq 3n+3$, then 
    $$
    H^i(X^{[1,2,3]}, \rho_{1,2,3}^* (M_{3, L_1} \otimes A_{3,L_2}) \otimes N') = 0~\text{and}~H^i(X^{[2,3]}, \rho_{2,3}^* (M_{3, L_1} \otimes A_{3,L_2}) \otimes N)=0~\text{for $i>0$}. 
    $$
\end{enumerate}
\end{prop}

\begin{proof}
$(1)$ This was shown in \cite[Lemma 4.2]{Choi.Lacini.Park.Sheridan.25}, but we include the proof for the reader's convenience. Let $\pr_2 \colon X \times X^{[1,2]} \to X^{[1,2]}$ be the second projection, and consider $\mathcal{W}_{1,2}=\{(x, (y, \xi)) \mid x,y \in \xi\} \subseteq X \times X^{[1,2]}$. As $L_1$ is very ample and $H^i(X, L_1)=0$ for $i>0$, we find
$$
R^i \pr_{2,*} (\mathscr{I}_{\mathcal{W}_{1,2}} \otimes (L_1 \boxtimes \mathcal{O}_{X^{[1,2]}})) = \begin{cases} \rho_{1,2}^* M_{2,L_1} & \text{for $i=0$} \\ 0 & \text{for $i>0$}. \end{cases}
$$
Then we have
$$
H^i(X^{[1,2]}, \rho_{1,2}^* (M_{2,L_1} \otimes N_{2,L_2}) \otimes N) = H^i(X \times X^{[1,2]}, \mathscr{I}_{\mathcal{W}_{1,2}} \otimes (L_1 \boxtimes (\rho_{1,2}^* N_{2,L_2} 
\otimes N)))~~\text{ for $i>0$}.  
$$
Recall that $X^{[1,2,3]}$ is obtained by the blow-up $\bl_{1,2,3} \colon X^{[1,2,3]} \to X \times X^{[1,2]}$ of $X \times X^{[1,2]}$ along $\mathcal{W}_{1,2}$ with exceptional divisor $F_{1,2}$ and $X^{[1,2,3]}$ has Gorenstein canonical singularities. Then
$$
H^i(X \times X^{[1,2]}, \mathscr{I}_{\mathcal{W}_{1,2}} \otimes (L_1 \boxtimes (\rho_{1,2}^* N_{2,L_2} \otimes N))) = H^i(X^{[1,2,3]}, \bl_{1,2,3}^* (L_1 \boxtimes (\rho_{1,2}^* N_{2,L_2} \otimes N))(-F_{1,2}))~~\text{ for $i>0$}.
$$
Since we may write
\begin{align*}
&\bl_{1,2,3}^* (L_1 \boxtimes (\rho_{1,2}^* N_{2,L_2} \otimes N))(-F_{1,2})\\
&= \omega_{X^{[1,2,3]}} \otimes \underbrace{\rho_{1,2,3}^* T_{3, H}}_{\text{nef and big}} \otimes \underbrace{\rho_{1,2,3}^* N_{3, H^2}^n \otimes \bl_{1,2,3}^*( (H^{m_1-2n-1} \otimes B_1) \boxtimes (\rho_{1,2}^* T_{2, H^{m_2-2n-1} \otimes B_2} \otimes N))}_{\text{nef}}.
\end{align*}
the first assertion follows from Kawamata--Viehweg vanishing.  As $\rho_{1,2}^* A_{2,L_2} = \rho_{1,2}^* N_{2, \omega_X \otimes H^{m_1-1} \otimes B_2} \otimes \rho_{1,2}^* N_{2, H}$ and $\rho_{1,2}^* N_{2, H}$ is nef, the first assertion implies the second assertion.

\medskip

\noindent $(2)$ Recall that there is a finite morphism $\rho_{1,2,3}' \colon X^{[1,2,3]} \to X^{[2,3]}$ and $X^{[1,2,3]}$ has Gorenstein canonical singularities. Note that $\rho_{1,2,3} = \rho_{2,3} \circ \rho_{1,2,3}'$ and $\mathcal{O}_{X^{[2,3]}}$ is a direct summand of $\rho_{1,2,3,*}' \mathcal{O}_{X^{[1,2,3]}}$. Thus we only need to prove the cohomology vanishing on $X^{[1,2,3]}$.
Let $\pr_2 \colon X \times X^{[1,2,3]} \to X^{[1,2,3]}$ be the second projection, and consider $\mathcal{W}_{1,2,3}=\{ (x, (y,\eta, \xi)) \mid x \in \xi, y \in \eta, \eta \subseteq \xi \} \subseteq X \times X^{[1,2,3]}$. As $L_1$ is $2$-very ample and $H^i(X, L_1) = 0$ for $i>0$, we find
$$
R^i \pr_{2,*} (\mathscr{I}_{\mathcal{W}_{1,2,3}} \otimes (L_1 \boxtimes \mathcal{O}_{X^{[1,2,3]}})) = \begin{cases} \rho_{1,2,3}^* M_{3,L_1} & \text{for $i=0$}\\ 0 & \text{for $i>0$}. \end{cases}
$$
Then we have
$$
H^i(X^{[1,2,3]}, \rho_{1,2,3}^* (M_{3, L_1} \otimes N_{3,L_2}) \otimes N')
= H^i(X \times X^{[1,2,3]}, \mathscr{I}_{\mathcal{W}_{1,2,3}} \otimes (L_1 \boxtimes (\rho_{1,2,3}^* N_{3, L_2} \otimes N')))~~\text{ for $i>0$}. 
$$
Recall that $X^{[1,2,3,4]}$ is obtained by the blow-up $\bl_{1,2,3,4} \colon X^{[1,2,3,4]} \to X \times X^{[1,2,3]}$ of $X \times X^{[1,2,3]}$ along $\mathcal{W}_{1,2,3}$ with exceptional divisor $F_{1,2,3}$ and $X^{[1,2,3,4]}$ has Gorenstein canonical singularities. Then 
$$
H^i(X \times X^{[1,2,3]}, \mathscr{I}_{\mathcal{W}_{1,2,3}} \otimes (L_1 \boxtimes (\rho_{1,2,3}^* N_{3, L_2} \otimes N')))
= H^i(X^{[1,2,3,4]}, \bl_{1,2,3,4}^* (L_1 \boxtimes (\rho_{1,2,3}^* N_{3, L_2} \otimes N'))(-F_{1,2,3}))
$$
for $i>0$. Since we may write
\begin{align*}
& \bl_{1,2,3,4}^* (L_1 \boxtimes (\rho_{1,2,3}^* N_{3, L_2} \otimes N'))(-F_{1,2,3})\\
& = \omega_{X^{[1,2,3,4]}} \otimes \underbrace{\rho_{1,2,3,4}^* T_{4, H}}_{\text{nef and big}} \otimes \underbrace{\rho_{1,2,3,4}^* N_{4, H^3}^n \otimes \bl_{1,2,3,4}^* ( (H^{m_1 - 3n-1} \otimes B_1) \boxtimes (\rho_{1,2,3}^* T_{3, H^{m_2 - 3n-1} \otimes B_2} \otimes N'))}_{\text{nef}},
\end{align*}
the first assertion follows from Kawamata--Viehweg vanishing. As $\rho_{1,2,3}^* A_{3, L_2} = \rho_{1,2,3}^* N_{3, \omega X \otimes H^{m-2} \otimes B_2} \otimes \rho_{1,2,3}^* N_{3, H^2}$ and $\rho_{1,2,3}^* N_{3, H^2}$ is nef, the first assertion implies the second assertion. 
\end{proof}

\begin{prop}\label{prop:mainvanishing2}
Let $L_1:=\omega_X \otimes H^{m_1} \otimes B_1$, $L_2:=\omega_X \otimes H^{m_2} \otimes B_2$, and $L_3:=\omega_X \otimes H^{m_3} \otimes B_3$ be line bundles on $X$ with $H$ very ample, $B_1, B_2, B_3$ nef for integers $m_1, m_2, m_3 \geq 0$. Then we have the following:
\begin{enumerate}
\item For a nef line bundle $N$ on $X^{[1,2]}$, if $m_1, m_2, m_3 \geq 3n+2$, then  
$$
H^i(X^{[1,2]}, \rho_{1,2}^* (M_{2,L_1} \otimes M_{2, L_2} \otimes A_{2, L_3}) \otimes N) = 0~~\text{ for $i>0$}. 
$$
\item If $m_1, m_2, m_3 \geq 3n+1$, then 
$$
H^i(X^{[2]}, M_{2,L_1} \otimes M_{2, L_2} \otimes N_{2, L_3}) = 0~~\text{ for $i>0$}. 
$$
\end{enumerate}
\end{prop}

\begin{proof}
$(1)$ Recall from the proof of Proposition \ref{prop:mainvanishing1} $(1)$ that
$$
H^i(X^{[1,2]}, \rho_{1,2}^* (M_{2,L_1} \otimes M_{2, L_2} \otimes A_{2, L_3}) \otimes N)
= H^i(X^{[1,2,3]}, \bl_{1,2,3}^*(L_1 \boxtimes (\rho_{1,2}^* (M_{2, L_2} \otimes A_{2,L_3}) \otimes N))(-F_{1,2}))
$$
for $i>0$. 
We have a short exact sequence
$$
0 \longrightarrow \rho_{1,2,3}^* M_{3, L_2} \longrightarrow \bl_{1,2,3}^* (\mathcal{O}_X \boxtimes \rho_{1,2}^* M_{2, L_2}) \longrightarrow \bl_{1,2,3}^* (L_2 \boxtimes \mathcal{O}_{X^{[1,2]}})(-F_{1,2}) \longrightarrow 0.
$$
It is enough to show that
\begin{align*}
&H^i(X^{[1,2,3]}, \rho_{1,2,3}^* M_{3, L_2} \otimes \bl_{1,2,3}^* (L_1 \boxtimes (\rho_{1,2}^* A_{2,L_3} \otimes N))(-F_{1,2}))=0\\
& H^i(X^{[1,2,3]}, \bl_{1,2,3}^* ( (L_1 \otimes L_2) \boxtimes (\rho_{1,2}^* A_{2,L_3} \otimes N))(-2F_{1,2}))=0.
\end{align*}
Since we may write
\begin{align*}
& \bl_{1,2,3}^* (L_1 \boxtimes (\rho_{1,2}^* A_{2,L_3} \otimes N))(-F_{1,2})\\
&= \rho_{1,2,3}^* N_{3, \omega_X \otimes H^{3n+1}} \otimes \underbrace{\bl_{1,2,3}^* ((H^{m_1 - 3n-1} \otimes B_1) \boxtimes (\rho_{1,2}^* N_{2, H^{m_3-3n-1} \otimes B_3} \otimes N))}_{\text{nef}}\\
& \bl_{1,2,3}^* ( (L_1 \otimes L_2) \boxtimes (\rho_{1,2}^* A_{2,L_3} \otimes N))(-2F_{1,2})\\
& = \rho_{1,2,3}^* A_{3, \omega_X \otimes H^{2n+3}} \otimes \underbrace{\bl_{1,2,3}^* ((H^{m_1 - 2n-3} \otimes B_1 \otimes L_2) \boxtimes (\rho_{1,2}^* T_{2, H^{m_3 - 2n-3} \otimes B_3} \otimes N))}_{\text{nef}},
\end{align*}
the assertion follows from Proposition \ref{prop:mainvanishing1} $(2)$ and Proposition \ref{prop:cohvanB} $(2)$.

\medskip

\noindent $(2)$ As $L_1$ is very ample and $H^i(X, L_1) = 0$ for $i>0$, we find
$$
R^i \tau_{2,3,*} (\res_{2,3}^* L_1 (-F_2)) = \begin{cases} M_{2, L_1} & \text{for $i=0$} \\ 0 & \text{for $i>0$}. \end{cases}
$$
Then we have
$$
H^i(X^{[2]}, M_{2,L_1} \otimes M_{2, L_2} \otimes N_{2, L_3})
= H^i(X^{[2,3]}, (\res_{2,3}^* L_1 \otimes \tau_{2,3}^* (M_{2, L_2} \otimes N_{2, L_3}))(-F_2))~~\text{ for $i>0$.}
$$
Consider the short exact sequence
$$
0 \longrightarrow \rho_{2,3}^* M_{3, L_2} \longrightarrow \tau_{2,3}^* M_{2, L_2} \longrightarrow \res_{2,3}^* L_2 (-F_2) \longrightarrow 0. 
$$
The desired cohomology vanishing can be deduced from
$$
H^i(X^{[2,3]}, \rho_{2,3}^* M_{3, L_2} \otimes (\res_{2,3}^* L_1 \otimes \tau_{2,3}^* N_{2, L_3})(-F_2))=0~\text{ and }~
H^i(X^{[2,3]}, (\res_{2,3}^* (L_1 \otimes L_2) \otimes \tau_{2,3}^* N_{2, L_3})(-2F_2))=0
$$
for $i>0$. As we may write
$$
(\res_{2,3}^* L_1 \otimes \tau_{2,3}^* N_{2, L_3})(-F_2) = \rho_{2,3}^* N_{3, \omega_X \otimes H^{3n+1}} \otimes \underbrace{\res_{2,3}^* (H^{m_1 - 3n-1} \otimes B_1) \otimes \tau_{2,3}^* T_{2, H^{m_3-3n-1} \otimes B_3}}_{\text{nef}},
$$
the first holds by Proposition \ref{prop:mainvanishing1} $(2)$. For the second, notice that $\rho_{1,2,3,*}' \mathcal{O}_{X^{[1,2,3]}} = \mathcal{O}_{X^{[2,3]}} \oplus \tau_{2,3}^* \mathcal{O}_{X^{[2]}}(-\delta_2)$. Thus it suffices to show that
$$
H^i(X^{[1,2,3]}, \res_{1,2,3}^* (L_1 \otimes L_2) \otimes \tau_{1,2,3}^* T_{2, L_3}(-2F_{1,2}))=0~~\text{ for $i>0$}.
$$
Recall that $X^{[1,2,3]}$ is obtained from $X^3$ by the composition of blow-ups
$$
b \colon X^{[1,2,3]} \xrightarrow{~\bl_{1,2,3}~} X \times X^{[1,2]} \xrightarrow{~\operatorname{id}_X \times \bl_{1,2}~} X^3.
$$
We may write
$$
\res_{1,2,3}^* (L_1 \otimes L_2) \otimes \tau_{1,2,3}^* T_{2, L_3}(-2F_{1,2})
= b^* ((L_1 \otimes L_2) \boxtimes L_3 \boxtimes L_3)(-2F_{1,2}).
$$
Note that 
$$
b_* \mathcal{O}_{X^{[1,2,3]}}(-2F_{1,2}) = (\operatorname{id}_X \times \bl_{1,2})_* \mathscr{I}_{\mathcal{W}_{1,2}}^2 = \mathscr{I}_{\Delta_{1,2} \cup \Delta_{1,3}}^2. 
$$
Then 
$$
H^i(X^{[1,2,3]}, \res_{1,2,3}^* (L_1 \otimes L_2) \otimes \tau_{1,2,3}^* T_{2, L_3}(-2F_{1,2}))
= H^i(X^3, \mathscr{I}_{\Delta_{1,2} \cup \Delta_{1,3}}^2 \otimes ((L_1 \otimes L_2) \boxtimes L_3 \boxtimes L_3))=0~~\text{ for $i>0$}
$$
by Proposition \ref{prop:cohvanA}.
\end{proof}

\section{Effective results for secant varieties}\label{sec:proofs}
\noindent This section is devoted to proving the main theorems. As always, $X$ is a smooth projective complex variety of dimension $n$. In the case $n=1$, better effective results than the bounds asserted in the main theorems have been established \cite{Agostini.Park.25, Ein.Niu.Park.20}, so we assume that $n \geq 2$.

\subsection{Proof of Theorem \ref{thm:main1}} 
We consider a line bundle $L = \omega_X \otimes H^m \otimes B$ on $X$ for an integer $m \geq 0$, where $H$ is very ample and $B$ is nef. Recall that if $m \geq n+1+p$, then $L$ is $p$-very ample. We assume that 
$$
\text{either $k=2$ and $m \geq 2n+2$ or $k=3$ and $m \geq 3n+3$}.
$$
In particular, $L$ is $(2k-1)$-very ample since $nk+k \geq n+2k$ for $k=2,3$ and $n \geq 2$. We consider the $k$-secant variety $\sigma_k=\sigma_k(X, L)$ of the embedding $X \subseteq \mathbb{P} H^0(X, L) = \mathbb{P}^r$. There is a birational morphism $\alpha_k \colon B^k \to \sigma_k$ given by the complete linear system $|\mathcal{O}_{B^{k}}(1)|$ of the tautological line bundle on $B^k = \mathbb{P}(E_{k,L})$ with canonical projection $\pi_k \colon B^k \to X^{[k]}$.

\begin{lem}\label{lem:maintechnical1}
For a nef line bundle $N$ on $X^{[k]}$ and an integer $\ell \geq 0$, we have the following:
\begin{enumerate}
    \item $H^i(B^k,  \pi_k^* (A_{k,L} \otimes N) \otimes \mathcal{O}_{B^k}(\ell))=0$ for $i>0$. 
    \item $H^i(B^k, \pi_k^* (M_{k,L} \otimes A_{k,L} \otimes N) \otimes \mathcal{O}_{B^k}(\ell))=0$ for $i>0$.
\end{enumerate}
\end{lem}

\begin{proof}
First, we prove the assertion $(1)$ for $k=2$. We proceed by induction on $\ell$ starting from $\ell = -1$. The assertion is trivial for $\ell=-1$. When $\ell=0$, the assertion is $H^i(X^{[2]}, A_{2,L} \otimes N)=0$ for $i>0$, which follow from Proposition \ref{prop:cohvanB} $(1)$ since $\mathcal{O}_{X^{[2]}}$ is a direct summand of $\rho_{1,2,*} \mathcal{O}_{X^{[1,2]}}$. Assume that $\ell \geq 1$, and consider the short exact sequence
$$
0 \longrightarrow \mathcal{O}_{B^2}(-Z_1^2) \longrightarrow \mathcal{O}_{B^2} \longrightarrow \mathcal{O}_{Z_1^2} \longrightarrow 0. 
$$
Recall that $\mathcal{O}_{B^2}(-Z_2^1) = \pi_2^* A_{2,L} \otimes \mathcal{O}_{B^2}(-2)$ and $Z_1^2 = X^{[1,2]}$. For $i>0$, the desired cohomology vanishing $H^i(B^2,  \pi_2^* (A_{2,L} \otimes N) \otimes \mathcal{O}_{B^2}(\ell))=0$ for $i>0$ is deduced from 
$$
H^i(B^2,  \pi_2^* (A_{2,L}^2 \otimes N) \otimes \mathcal{O}_{B^2}(\ell-2))=0~~\text{ and }~~H^i(X^{[1,2]}, \rho_{1,2}^* (A_{2,L} \otimes N) \otimes \res_{1,2}^* L^{\ell})=0~~\text{ for $i>0$}.
$$
The former holds by induction, and the latter holds by Proposition \ref{prop:cohvanB} $(1)$. We have shown the assertion $(1)$ for $k=2$. The assertion $(2)$ for $k=2$ can be proven similarly using Proposition \ref{prop:mainvanishing1} $(1)$ instead. 

\medskip

\noindent Next, we prove the assertion $(1)$ for $k=3$. We proceed by induction on $\ell$ starting from $\ell=-2$. The assertion is trivial for $\ell=-2$ and $-1$. When $\ell=0$, the assertion is $H^i(X^{[3]}, A_{3,L} \otimes N)=0$ for $i>0$, which follow from Proposition \ref{prop:cohvanB} $(2)$ since $\mathcal{O}_{X^{[3]}}$ is a direct summand of $\rho_{2,3,*} \mathcal{O}_{X^{[2,3]}}$. Assume that $\ell \geq 1$, and consider the short exact sequences
$$
0 \longrightarrow \mathcal{O}_{B^3}(-Z_2^3) \longrightarrow \mathcal{O}_{B^3} \longrightarrow \mathcal{O}_{Z_2^3} \longrightarrow 0~\text{ and }~
0 \longrightarrow \alpha_{2,3,*} \mathcal{O}_{B^{2,3}}(-\widetilde{\tau}_{1,2}^* Z_1^2) \longrightarrow \mathcal{O}_{Z_2^3} \longrightarrow \mathcal{O}_{Z_1^3} \longrightarrow 0.
$$
Recall that $\mathcal{O}_{B^3}(-Z_2^3) = \pi_3^* A_{3,L} \otimes \mathcal{O}_{B^3}(-3)$ and $B^{2,3}=\mathbb{P}(\tau_{2,3}^* E_{2,L})$. There are the canonical projection $\pi_{2,3} \colon B^{2,3} \to X^{[2,3]}$ and a natural morphism $\widetilde{\tau}_{2,3} \colon B^{2,3} \to B^2$. Recall that the morphism $\alpha_{2,3} \colon B^{2,3} \to Z_2^3$ is birational and $R^i \alpha_{2,3,*} \mathcal{O}_{B^{2,3}}(-\widetilde{\tau}_{2,3}^* Z_1^2) = 0$ for $i>0$. In addition, $Z_1^3 = \mathcal{Z}_3$ has rational singularities, and the natural morphism $X^{[2,3]} \to \mathcal{Z}_3$ is a resolution of singularities. For $i>0$, the desired cohomology vanishing $H^i(B^3,  \pi_3^* (A_{3,L} \otimes N) \otimes \mathcal{O}_{B^3}(\ell))=0$ for $i>0$ is deduced from 
\begin{align*}
& H^i(B^3,  \pi_3^* (A_{3,L}^2 \otimes N) \otimes \mathcal{O}_{B^3}(\ell-3))=0,~~ H^i(B^{2,3}, \alpha_{2,3}^* \pi_3^* (A_{3,L} \otimes N) \otimes \mathcal{O}_{B^{2,3}}(\ell)(-\widetilde{\tau}_{2,3}^* Z_1^2)) = 0,~~\text{and}\\
& H^i(X^{[2,3]}, \rho_{2,3}^* (A_{3,L} \otimes N) \otimes \res_{2,3}^*L^{\ell})=0.
\end{align*}
The first holds by induction, and the third holds by Proposition \ref{prop:cohvanB} $(2)$. Recall that $\mathcal{O}_{B^{2,3}}(-\widetilde{\tau}_{2,3}^* Z_1^2) = \pi_{2,3}^* \tau_{2,3}^* A_{2,L} \otimes \mathcal{O}_{B^{2,3}}(-2)$. The second can be rewritten as 
$$
H^i(B^{2,3}, \pi_{2,3}^* (\rho_{2,3}^*(A_{3,L} \otimes N) \otimes \tau_{2,3}^* A_{2,L}) \otimes \mathcal{O}_{B^{2,3}}(\ell-2))=0~~\text{ for $i>0$}.
$$
It suffices to establish the claim that if $a \geq 1$ and $b \geq -1$ are integers, then
$$
H^i(B^{2,3}, \pi_{2,3}^* (\rho_{2,3}^*(A_{3,L} \otimes N) \otimes \tau_{2,3}^* A_{2,L}^a) \otimes \mathcal{O}_{B^{2,3}}(b))=0~~\text{ for $i>0$}.
$$
To this end, we proceed by induction on $b$. When $b=-1$, the claim is trivial. When $b=0$, the claim is $H^i(X^{[2,3]}, \rho_{2,3}^*(A_{3,L} \otimes N) \otimes \tau_{2,3}^* A_{2,L}^a)=0$, which follows from Proposition \ref{prop:cohvanB} $(2)$. Assume that $b \geq 1$, and consider the short exact sequence
$$
0 \longrightarrow \mathcal{O}_{B^{2,3}}(-\widetilde{\tau}_{2,3}^* Z_1^2) \longrightarrow \mathcal{O}_{B^{2,3}} \longrightarrow \mathcal{O}_{\widetilde{\tau}_{2,3}^* Z_1^2} \longrightarrow 0. 
$$
Recall that $\widetilde{\tau}_{2,3}^* Z_1^2 = X^{[1,2,3]}$ and there is a  morphism $\rho_{1,2,3}' \colon X^{[1,2,3]} \to X^{[2,3]}$. The claim is deduced from 
\begin{align*}
&H^i(B^{2,3}, \pi_{2,3}^* (\rho_{2,3}^*(A_{3,L} \otimes N) \otimes \tau_{2,3}^* A_{2,L}^{a+1}) \otimes \mathcal{O}_{B^{2,3}}(b-2))=0\\
& H^i(X^{[1,2,3]}, \rho^{\prime *}_{1,2,3} (\rho_{2,3}^* (A_{3,L} \otimes N) \otimes \tau_{2,3}^* A_{2,L}^a \otimes \res_{2,3}^* L^b) )=0.
\end{align*}
The former holds by induction, and the latter holds by Proposition \ref{prop:cohvanB} $(2)$ since
$$
\rho^{\prime *}_{1,2,3}(\rho_{2,3}^* (A_{3,L} \otimes N) \otimes \tau_{2,3}^* A_{2,L}^a \otimes \res_{2,3}^* L^b)
= \rho_{1,2,3}^* (A_{3, L} \otimes N) \otimes \underbrace{\bl_{1,2,3} (L^b \boxtimes \rho_{1,2}^* A_{2,L}^a)}_{\text{nef}},
$$
where $\bl_{1,2,3} \colon X^{[1,2,3]} \to X \times X^{[1,2]}$ is the blow-up of $X \times X^{[1,2]}$ along $\mathcal{W}_{1,2}$. We have shown the assertion $(1)$ for $k=3$. The assertion $(2)$ for $k=3$ can be proven similarly using Proposition \ref{prop:mainvanishing1} $(2)$ instead. 
\end{proof}

\begin{thm}\label{thm:maincohvan1}
We have
$$
H^i(X^{[k]}, \wedge^j M_{k,L} \otimes A_{k,L})=0~~\text{ for $i>0$ and $0 \leq j \leq i$}.
$$
\end{thm}

\begin{proof}
We have a commutative diagram with exact sequences
\begin{equation}\label{eq:commdiagonB^k}\tag{$*$}
\begin{gathered}
\xymatrix{
& & &0\ar[d] & \\
& 0\ar[d] & & K_k \ar[d] &\\
0\ar[r]& \pi_k ^* M_{k,L}\ar[r]\ar[d]& H^0(X,L)\otimes \mathcal{O}_{B^k}\ar[r]\ar@{=}[d]& \pi_k ^* E_{k,L}\ar[r]\ar[d]& 0\\
0\ar[r]& M_{\mathcal{O}_{B^k}(1)}\ar[r]\ar[d]& H^0(X,L)\otimes \mathcal{O}_{B^k}\ar[r]& \mathcal{O}_{B^k}(1)\ar[r]\ar[d]& 0\\
& K_k \ar[d] & & ~0. & \\
& 0 & & &
}
\end{gathered}
\end{equation}
As the right vertical short exact sequence is a relative Euler sequence, Bott vanishing shows that $R^i \pi_{k,*} \wedge^j K_k = 0$ for $i > 0$ and $j \geq 0$ (see also \cite[Lemma 6.1]{Choi.Lacini.Park.Sheridan.25}). Then 
$$
H^i(X^{[k]}, \wedge^j M_{k,L} \otimes A_{k,L})= H^i(B^k, \wedge^j M_{\mathcal{O}_{B^k}(1)} \otimes \pi_k^* A_{k,L})~~\text{ for $i>0$ and $j \geq 0$}. 
$$
Considering the short exact sequence
$$
0 \longrightarrow \wedge^j M_{\mathcal{O}_{B^k}(1)} \longrightarrow \wedge^j H^0(X, L) \otimes \mathcal{O}_{B^k} \longrightarrow \wedge^{j-1} M_{\mathcal{O}_{B^k}(1)} \otimes \mathcal{O}_{B^k}(1) \longrightarrow 0
$$
and applying Lemma \ref{lem:maintechnical1} $(1)$ successively, we see that
$$
H^i(B^k, \wedge^j M_{\mathcal{O}_{B^k}(1)} \otimes \pi_k^* A_{k,L}) = \begin{cases} H^1(B^k, M_{\mathcal{O}_{B^k}(1)} \otimes \mathcal{O}_{B^k}(i-1) \otimes \pi_k^* A_{k,L}) & \text{for $i>0$ and $j=i$} \\
H^{i-j}(B^k, \mathcal{O}_{B^k}(j) \otimes \pi_k^* A_{k,L}) & \text{for $i>0$ and $0 \leq j \leq i-1$}. \end{cases}
$$
When $i>0$ and $0 \leq j \leq i-1$, we obtain $H^i(B^k, \wedge^j M_{\mathcal{O}_{B^k}(1)} \otimes \pi_k^* A_{k,L})=0$ from Lemma \ref{lem:maintechnical1} $(1)$. We henceforth assume that $j=i$. By Lemma \ref{lem:maintechnical1} $(1)$, the cohomology vanishing $H^1(B^k, M_{\mathcal{O}_{B^k}(1)} \otimes \mathcal{O}_{B^k}(i-1) \otimes \pi_k^* A_{k,L})=0$ for $i>0$  is equivalent to the surjectivity of the multiplication map
$$
H^0(B^k, \mathcal{O}_{B^k}(1)) \otimes H^0(B^k, \pi_k^* A_{k,L} \otimes \mathcal{O}_{B^k}(i-1)) \longrightarrow H^0(B^k, \pi_k^* A_{k,L} \otimes \mathcal{O}_{B^k}(i))
$$
for $i>0$. This map fits into a commutative diagram
$$
\xymatrix{
H^0(X, L) \otimes H^0(B^k, \pi_k^* A_{k,L} \otimes \mathcal{O}_{B^k}(i-1)) \ar[r] \ar@{=}[d] & H^0(B^k, \pi_k^* (E_{k,L} \otimes A_{k,L}) \otimes \mathcal{O}_{B^k}(i-1)) \ar[d] \\
H^0(X, L) \otimes H^0(B^k, \pi_k^* A_{k,L} \otimes \mathcal{O}_{B^k}(i-1)) \ar[r] & H^0(B^k, \pi_k^* A_{k,L} \otimes \mathcal{O}_{B^k}(i)).
}
$$
The right vertical map can be identified with the map
$$
H^0(X^{[k]}, S^{i-1} E_{k,L} \otimes E_{k,L} \otimes A_{k,L})\longrightarrow H^0(X^{[k]}, S^i E_{k,L} \otimes A_{k,L}). 
$$
This map evidently splits, so in particular, this map is surjective. Thus it suffices to check that the upper horizontal map in the above commutative diagram is surjective. By Lemma \ref{lem:maintechnical1} $(1)$, this is equivalent to the cohomology vanishing
$$
H^1(B^k, \pi_k^* ( M_{k,L} \otimes A_{k,L}) \otimes \mathcal{O}_{B^k}(i-1)) = 0~~\text{ for $i>0$}.
$$
But this holds by Lemma \ref{lem:maintechnical1} $(2)$.
\end{proof}

As was noticed in \cite[Remark 5.6]{Choi.Lacini.Park.Sheridan.25}, Theorem \ref{thm:maincohvan1} implies Theorem \ref{thm:main1}. Here we briefly explain how the implication works (see \cite[Section 5]{Choi.Lacini.Park.Sheridan.25} for more details). We keep assuming $k=2$ or $k=3$. To prove that $\sigma_k$ is normal and $\sigma_k \subseteq \mathbb{P}^r$ is projectively normal, we need to show that the map
$$
\varphi_{\ell} \colon S^{\ell} H^0(X, L) \longrightarrow H^0(B^k, \mathcal{O}_{B^k}(\ell))
$$
is surjective for $\ell \geq 1$. To this end, consider the short exact sequence
$$
0 \longrightarrow \mathcal{O}_{B^k}(-Z_{k-1}^k) \longrightarrow \mathcal{O}_{B^k} \longrightarrow \mathcal{O}_{Z_{k-1}^k} \longrightarrow 0. 
$$
Recall that $\mathcal{O}_{B^k}(-Z_{k-1}^k) = \pi_k^* A_{k,L} \otimes \mathcal{O}_{B^k}(-k)$. As there is an injective map $\alpha_{k,*} \mathcal{O}_{Z_{k-1}^k} \hookrightarrow \alpha_{k,*} \alpha_{k-1,k,*} \mathcal{O}_{B^{k-1,k}}$ and $\alpha_{k,*} \alpha_{k-1,k,*} \mathcal{O}_{B^{k-1,k}} = \alpha_{k-1,*} \widetilde{\tau}_{k-1,k,*} \mathcal{O}_{B^{k-1,k}} = \mathcal{O}_{\sigma_{k-1}}$, we get $\alpha_{k,*} \mathcal{O}_{Z_{k-1}^k} = \mathcal{O}_{\sigma_{k-1}}$ (see \cite[Lemma 5.1]{Choi.Lacini.Park.Sheridan.25}). We may assume that $\sigma_{k-1}=\sigma_{k-1}(X, L) \subseteq \mathbb{P}^r$ is projectively normal. On the other hand, by Danila's theorem \cite{Danila.07} (see also \cite[Theorem 2.8]{Choi.Lacini.Park.Sheridan.25}), the map $\varphi_{\ell}$ is an isomorphism for $1 \leq \ell \leq k$. Thus it is enough to check that
$$
S^{\ell} H^0(X, L) \otimes H^0(B^k, \mathcal{O}_{B^k}(k) \otimes \mathcal{O}_{B^k}(-Z_{k-1}^k)) \longrightarrow H^0(B^k, \mathcal{O}_{B^k}(k+\ell) \otimes \mathcal{O}_{B^k}(-Z_{k-1}^k))
$$
is surjective for $\ell \geq 1$. This map can be identified with
$$
S^{\ell} H^0(X, L) \otimes H^0(X^{[k]}, A_{k, L}) \longrightarrow H^0(X^{[k]}, S^{\ell} E_{k,L} \otimes A_{k,L}). 
$$
Since there is a long exact sequence
$$
\cdots \longrightarrow \wedge^2 M_{k,L} \otimes S^{\ell-2} H^0(X, L) \longrightarrow M_{k,L} \otimes S^{\ell-1} H^0(X, L) \longrightarrow S^{\ell} H^0(X, L) \otimes \mathcal{O}_{X^{[k]}} \longrightarrow S^{\ell} E_{k,L} \longrightarrow 0,
$$
the assertion follows from Theorem \ref{thm:maincohvan1}. We obtain the assertion $(1)$ of Theorem \ref{thm:main1}. Note also that Theorem \ref{thm:maincohvan1} shows that
$$
H^i(B^k, \mathcal{O}_{B^k}(\ell) \otimes \mathcal{O}_{B^k}(-Z_{k-1}^k))=H^i(X^{[k]}, S^{\ell-k} E_{k,L} \otimes A_{k,L})=0~~\text{ for $i>0$ and $\ell \geq 1$}. 
$$
Notice that $\mathcal{O}_{B^k}(1) = \alpha_k^* \mathcal{O}_{\sigma_k}(1)$. Thus we obtain the Du Bois-type condition
$$
R^i \alpha_{k,*} \mathcal{O}_{B^k}(-Z_{k-1}^k) = \begin{cases} \mathscr{I}_{\sigma_{k-1}/\sigma_k} & \text{for $i=0$} \\ 0 & \text{for $i>0$} ,\end{cases}
$$
and $H^i(\sigma_k, \mathscr{I}_{\sigma_{k-1}/\sigma_k}(\ell))=0$ for $i>0$ and $\ell \geq 1$. As we may assume that $H^i(\sigma_{k-1}, \mathcal{O}_{\sigma_{k-1}}(\ell))=0$ for $i>0$ and $\ell \geq 1$, we get $H^i(\sigma_k, \mathcal{O}_{\sigma_k}(\ell))=0$ for $i>0$ and $\ell \geq 1$ (see \cite[Theorem 5.4]{Choi.Lacini.Park.Sheridan.25}). This is the assertion $(3)$ of Theorem \ref{thm:main1}. Recall that $B^k$ is smooth and $Z_{k-1}^k$ has Du Bois singularities. We may assume that $\sigma_{k-1}$ has also Du Bois singularities. By Koll\'{a}r--Kov\'{a}cs criterion \cite[Corollary 6.28]{Kollar.13}, the Du Bois-type condition implies that $\sigma_k$ has Du Bois singularities. We have shown the assertion $(1)$ of Theorem \ref{thm:main1}.
The assertion $(4)$ of Theorem \ref{thm:main1} is shown in \cite[Subsections 5.3 and 5.4]{Choi.Lacini.Park.Sheridan.25}. However, in our cases $k=2$ or $3$, the proof is much simpler. If $H^i(X, \mathcal{O}_X) \neq 0$ for some $1 \leq i \leq n-1$, then it is easy to confirm that $\sigma_k$ is not Cohen--Macaulay (see \cite[Proposition 5.19]{Choi.Lacini.Park.Sheridan.25}). The difficult part is to show that if $H^i(X, \mathcal{O}_X)=0$ for $1 \leq i \leq n-1$, then 
$$
R^i \alpha_{k,*} \omega_{B^k}(Z_{k-1}^k) = 
\begin{cases} \omega_{\sigma_k} & \text{for $i=0$} \\ \omega_{\sigma_{k-1}} & \text{for $i=n$} \\ 0 & \text{otherwise}. \end{cases}
$$
This easily follows by considering the short exact sequence
$$
0 \longrightarrow \omega_{B^k} \longrightarrow \omega_{B^k}(Z_{k-1}^k) \longrightarrow \omega_{Z_{k-1}^k} \longrightarrow 0. 
$$
For the assertion $(5)$ of Theorem \ref{thm:main1}, it is sufficient to show that $\alpha_{k,*} \omega_{B^k} = \alpha_{k,*} \omega_{B^k}(Z_{k-1}^k)$ if and only if $H^0(X, \omega_X)=0$ under the assumption that $H^i(X, \mathcal{O}_X)=0$ for $1 \leq i \leq n-1$ (equivalently, $\sigma_k$ is Cohen--Macaulay). One can readily check that both conditions are equivalent to $\alpha_{k,*} \omega_{Z_{k-1}^k}=0$ (see \cite[Theorem 5.29]{Choi.Lacini.Park.Sheridan.25}).

\subsection{Proof of Theorem \ref{thm:main2}}
We keep using the notations in the previous subsection, but we assume that $k=2$ and $m \geq 3n+2$. 

\begin{lem}\label{lem:maintechnical2}
For a nef line bundle $N$ on $X^{[2]}$ and an integer $\ell \geq 0$, we have
$$
H^i(B^2, \pi_2^*(\wedge^2 M_{2, L} \otimes A_{2,L} \otimes N) \otimes \mathcal{O}_{B^2}(\ell))=0~~\text{ for $i>0$}.
$$
\end{lem}

\begin{proof}
The proof is identical to the proof of Lemma \ref{lem:maintechnical1}, but we include the whole proof for the reader's convenience. We proceed by induction on $\ell$ starting from $\ell=-1$. The assertion is trivial for $\ell = -1$. hen $\ell=0$, the assertion is $H^i(X^{[2]}, \wedge^2 M_{2, L} \otimes A_{2,L} \otimes N)=0$ for $i>0$, which follow from Proposition \ref{prop:mainvanishing2} $(1)$ since $\mathcal{O}_{X^{[2]}}$ is a direct summand of $\rho_{1,2,*} \mathcal{O}_{X^{[1,2]}}$ and $\wedge^2 M_{2, L}$ is a direct summand of $M_{2, L} \otimes M_{2,L}$. Assume that $\ell \geq 1$, and consider the short exact sequence
$$
0 \longrightarrow \mathcal{O}_{B^2}(-Z_1^2) \longrightarrow \mathcal{O}_{B^2} \longrightarrow \mathcal{O}_{Z_1^2} \longrightarrow 0. 
$$
Recall that $\mathcal{O}_{B^2}(-Z_2^1) = \pi_2^* A_{2,L} \otimes \mathcal{O}_{B^2}(-2)$ and $Z_1^2 = X^{[1,2]}$. For $i>0$, the desired cohomology vanishing $H^i(B^2,  \pi_2^* (\wedge^2 M_{2, L} \otimes A_{2,L} \otimes N) \otimes \mathcal{O}_{B^2}(\ell))=0$ for $i>0$ is deduced from 
$$
H^i(B^2,  \pi_2^* (\wedge^2 M_{2, L} \otimes A_{2,L}^2 \otimes N) \otimes \mathcal{O}_{B^2}(\ell-2))=0~~\text{ and }~~H^i(X^{[1,2]}, \rho_{1,2}^* (\wedge^2 M_{2, L} \otimes A_{2,L} \otimes N) \otimes \res_{1,2}^* L^{\ell})=0~~\text{ for $i>0$}.
$$
The former holds by induction, and the latter holds by Proposition \ref{prop:mainvanishing2} $(1)$.
\end{proof}

\begin{thm}\label{thm:maincohvan2}
We have
$$
H^i(X^{[2]}, \wedge^j M_{2,L} \otimes A_{2,L})=0~~\text{ for $i>0$ and $0 \leq j \leq i+1$}. 
$$
\end{thm}

\begin{proof}
In view of Theorem \ref{thm:maincohvan1}, we only need to consider the case $j=i+1$. Recall that
$$
H^i(X^{[2]}, \wedge^{i+1} M_{2,L} \otimes A_{2,L}) = H^i(B^2, \wedge^{i+1} M_{\mathcal{O}_{B^2}(1)} \otimes \pi_2^* A_{2, L})~~\text{ for $i>0$.}
$$
Considering the short exact sequence
$$
0 \longrightarrow \wedge^j M_{\mathcal{O}_{B^k}(1)} \longrightarrow \wedge^j H^0(X, L) \otimes \mathcal{O}_{B^k} \longrightarrow \wedge^{j-1} M_{\mathcal{O}_{B^k}(1)} \otimes \mathcal{O}_{B^k}(1) \longrightarrow 0
$$
and applying Lemma \ref{lem:maintechnical1} $(1)$ successively, we see that
$$
H^i(B^2, \wedge^{i+1} M_{\mathcal{O}_{B^2}(1)} \otimes \pi_2^* A_{2, L})
= H^1(B^2, \wedge^2 M_{\mathcal{O}_{B^2}(1)} \otimes \mathcal{O}_{B^2}(i-1) \otimes \pi_2^* A_{2,L})~~\text{ for $i>0$}. 
$$
Note that $K_2$ in the commutative diagram (\ref{eq:commdiagonB^k}) is the line bundle $\pi_2^* N_{2,L} \otimes \mathcal{O}_{B^2}(-1)$. Then we have a short exact sequence
$$
0 \longrightarrow  \pi_2^* \wedge^2 M_{2,L} \longrightarrow \wedge^2 M_{\mathcal{O}_{B^k}} \longrightarrow \pi_2^* (M_{2,L} \otimes N_{2,L}) \otimes \mathcal{O}_{B^2}(-1) \longrightarrow 0. 
$$
It suffices to check that
$$
H^1(B^2,  \pi_2^* (\wedge^2 M_{2,L} \otimes A_{2,L}) \otimes \mathcal{O}_{B^2}(\ell))=0~\text{ and }~H^1(B^2, \pi_2^* (M_{2,L} \otimes A_{2,L} \otimes N_{2,L}) \otimes \mathcal{O}_{B^2}(\ell-1))=0~\text{for $\ell \geq 0$}.
$$
The former holds by Lemma \ref{lem:maintechnical2}, and the latter holds by Lemma \ref{lem:maintechnical1} $(2)$.
\end{proof}

As was noticed in \cite[Remark 6.4]{Choi.Lacini.Park.Sheridan.25}, Theorem \ref{thm:maincohvan2} implies Theorem \ref{thm:main2}. Here we briefly explain how the implication works (see \cite[Section 6]{Choi.Lacini.Park.Sheridan.25} for more details). By Theorem \ref{thm:main1}, we know that $\sigma_2=\sigma_2(X, L) \subseteq \mathbb{P}^r$ is projectively normal. It is well-known that $I(\sigma_2(X, L))_2 = 0$.  
Thus $I(\sigma_2(X, L))$ is generated by cubics if and only if $K_{1,q}(\sigma_2, \mathcal{O}_{\sigma_2}(1))=0$ for $q \geq 3$. We refer to \cite[Lecture 5]{Ein.Lazarsfeld.06} for a general overview on syzygies and Koszul cohomology. 
Consider the short exact sequence
$$
0 \longrightarrow \mathscr{I}_{X/\sigma_2} \longrightarrow \mathcal{O}_{\sigma_2} \longrightarrow \mathcal{O}_X \longrightarrow 0. 
$$
By Proposition \ref{prop:cohvanA}, $K_{1, q}(X, L) = H^{q-2}(X, \wedge^{q-1} M_L \otimes L^2)=H^1(X, \wedge^2 M_L \otimes L^{q-1})=0$ for $q \geq 3$. 
Thus 
$$
K_{1,q}(\sigma_2, \mathcal{O}_{\sigma_2}(1))=H^{q-2}(\sigma_2, \wedge^{q-1} M_{\mathcal{O}_{\sigma_2}(1)} \otimes \mathcal{O}_{\sigma_2}(2))=0~~\text{ for $q \geq 3$}
$$
can be deduced from 
$$
H^{q-2}(\sigma_2, \wedge^{q-1} M_{\mathcal{O}_{\sigma_2}(1)} \otimes  \mathscr{I}_{\sigma_1/\sigma_2}(2))
= H^{q-2}(B^2, \wedge^{q-1} M_{\mathcal{O}_{B^2}(1)} \otimes \pi_2^* A_{2,L})
= H^{q-2}(X^{[2]}, \wedge^{q-1} M_{2,L} \otimes A_{2,L})=0
$$
for $q \geq 3$. But this holds by Theorem \ref{thm:maincohvan2}.

\subsection{Proof of Theorem \ref{thm:main3}}
Let $L_1:=\omega_X \otimes H^{m_1} \otimes B_1$ and $L_2:=\omega_X \otimes H^{m_2} \otimes B_2$ be line bundles on $X$ with $H$ very ample and $B_1, B_2$ nef for integers $m_1, m_2 \geq 0$, and $L:=L_1 \otimes L_2$. We assume that $m_1, m_2 \geq 3n+2$, and we consider the $2$-secant variety $\sigma_2=\sigma_2(X, L)$ of $X \subseteq \mathbb{P} H^0(X, L) = \mathbb{P}^r$. 
As $L_2$ is nef, Theorems \ref{thm:main1} and \ref{thm:main2} show that $\sigma_2 \subseteq \mathbb{P}^r$ is projectively normal and the homogeneous ideal $I(\sigma_2(X, L))$ is generated by cubics.

\begin{lem}\label{lem:maintechnical3}
We have
$$
H^2(X^{[2]}, S^2 M_{2, L_1} \otimes M_{2, L_1 \otimes L_2} \otimes N_{2, L_2})=0.
$$
\end{lem}

\begin{proof}
Considering the short exact sequence
$$
0 \longrightarrow \rho_{1,2}^* S^2 M_{2, L_1} \longrightarrow \tau_{1,2}^* S^2 M_{L_1} \longrightarrow (\tau_{1,2}^* M_{L_1} \otimes \res_{1,2}^* L_1)(-F_1) \longrightarrow 0,
$$
we can derive the lemma from 
\begin{align*}
& H^1(X^{[1,2]}, \rho_{1,2}^* (M_{2, L_1 \otimes L_2} \otimes A_{2, L_2}) \otimes  \tau_{1,2}^* M_{L_1} \otimes \res_{1,2}^* L_1) = 0\\
& H^2(X^{[1,2]}, \rho_{1,2}^* (M_{2, L_1 \otimes L_2} \otimes N_{2, L_2}) \otimes \tau_{1,2}^* S^2 M_{L_1} )=0.
\end{align*}
For the first cohomology vanishing, considering the short exact sequence
$$
0 \longrightarrow \rho_{1,2}^* M_{2, L_1} \longrightarrow \tau_{1,2}^* M_{L_1} \longrightarrow \res_{1,2}^* L_1(-F_1) \longrightarrow 0, 
$$
we reduce the problem to 
\begin{align*}
& H^1 (X^{[1,2]}, \rho_{1,2}^* (M_{2, L_1} \otimes M_{2, L_1 \otimes L_1} \otimes A_{2,L_2}) \otimes \res_{1,2}^* L_1)=0\\
& H^1(X^{[1,2]}, \rho_{1,2}^* (M_{2, L_1 \otimes L_2} \otimes A_{2, L_2}(-\delta_2)) \otimes \res_{1,2}^*L_1^2)=0.
\end{align*}
These hold by Proposition \ref{prop:mainvanishing2} $(1)$ and Proposition \ref{prop:mainvanishing1} $(1)$. 
For the second cohomology vanishing, considering the short exact sequence
$$
0 \longrightarrow \tau_{1,2}^* S^2 M_{L_1} \longrightarrow S^2 H^0(X, L_1) \otimes \mathcal{O}_{X^{[1,2]}} \longrightarrow H^0(X, L_1) \otimes \tau_{1,2}^* L_1 \longrightarrow 0, 
$$
we reduce the problem to 
$$
H^1(X^{[1,2]}, \rho_{1,2}^* (M_{2, L_1 \otimes L_2} \otimes N_{2, L_2}) \otimes \tau_{1,2}^* L_1)=0
~\text{ and }~
H^2(X^{[1,2]}, \rho_{1,2}^*(M_{2, L_1 \otimes L_2} \otimes N_{2, L_2}))=0.
$$
These hold by Proposition \ref{prop:mainvanishing1} $(1)$.
\end{proof}

We consider the linear map 
$$
m_{L_1, L_2}^3 \colon \wedge^3 H^0(X, L_1) \otimes \wedge^3 H^0(X, L_2) \longrightarrow I(\sigma_2(X, L_1 \otimes L_2))_3
$$
given by 
$$
(s_1 \wedge s_2 \wedge s_3) \otimes (t_1 \wedge t_2 \wedge t_3) \longmapsto \sum_{\sigma \in \mathfrak{S}_3} \operatorname{sign}(\sigma) (s_{\sigma(1)} t_1) (s_{\sigma(2)} t_2) (s_{\sigma(3)} t_3). 
$$
Recall that the image of $m_{L_1, L_2}^3$ is spanned by $3 \times 3$-minors of the catalecticant matrix $\operatorname{Cat}(L_1, L_2)$. To prove Theorem \ref{thm:main3}, we only need to show that the map $m_{L_1, L_2}^3$ is surjective. This map can be identified with the multiplication map
$$
m_{L_1, L_2}^3 \colon H^0(X^{[3]}, N_{3, L_1}) \otimes H^0(X^{[3]}, N_{3,L_2}) \longrightarrow H^0(X^{[3]}, A_{3, L_1 \otimes L_2}). 
$$
We refer to \cite{Agostini.Park.25} for more details about the map $m_{L_1, L_2}^3$ and the related geometry of Hilbert schemes of points. 
As the evaluation map $H^0(X^{[3]}, N_{3, L_1}) \otimes \mathcal{O}_{X^{[3]}} \to N_{3, L_1}$ is a direct summand of the pushforward of the evaluation map $H^0(X, L_1) \otimes H^0(X^{[2]}, N_{2, L_1}) \otimes \mathcal{O}_{X^{[2,3]}} \to \res_{1,2}^* L_1 \otimes \tau_{1,2}^* N_{2,L_1}$ via $\rho_{1,2}$, it is enough to show that
$$
H^1(X^{[2,3]}, M_{\res_{1,2}^* L_1 \otimes \tau_{1,2}^* N_{2,L_1}} \otimes \rho_{1,2}^* N_{3, L_2})=0. 
$$
Note that $\rho_{1,2}^* N_{3, L_2} = (\res_{1,2}^* L_2 \otimes \tau_{1,2}^* N_{2, L_2})(-F_2)$. We have a commutative diagram with exact sequences
$$
\xymatrixrowsep{0.3in}
\xymatrixcolsep{0.21in}
\xymatrix{
& 0 \ar[d] & 0 \ar[d] & &\\
& \res_{1,2}^* M_{L_1} \otimes H^0(X^{[2]}, N_{2, L_1}) \ar[d] \ar@{=}[r] & \res_{1,2}^* M_{L_1} \otimes H^0(X^{[2]}, N_{2, L_1})  \ar[d] & & \\
0 \ar[r] &  M_{\res_{1,2}^* L_1 \otimes \tau_{1,2}^* N_{2, L_1}} \ar[d] \ar[r] & H^0(X, L_1) \otimes H^0(X^{[2]}, N_{2, L_1}) \otimes \mathcal{O}_{x^{[2,3]}} \ar[r] \ar[d] & \res_{1,2}^* L_1 \otimes \tau_{1,2}^* N_{2, L_1} \ar@{=}[d] \ar[r] & 0 \\
0 \ar[r] & \res_{1,2}^* L_1 \otimes \tau_{1,2}^* M_{N_{2, L_1}} \ar[r] \ar[d] & \res_{1,2}^* L_1 \otimes H^0(X^{[2]}, N_{2, L_1}) \ar[r] \ar[d] & \res_{1,2}^* L_1 \otimes \tau_{1,2}^* N_{2, L_1} \ar[r] & 0  \\
& 0 & ~0. & &
}
$$
From the left vertical short exact sequence, we reduce the problem to 
\begin{align*}
& H^1(X^{[2,3]}, (\res_{1,2}^* (M_{L_1} \otimes L_2) \otimes \tau_{1,2}^* N_{2, L_2})(-F_2) ) = 0\\
& H^1(X^{[2,3]}, (\res_{1,2}^* (L_1 \otimes L_2) \otimes \tau_{1,2}^* (M_{N_2, L_1} \otimes N_{2, L_2})(-F_2))=0. 
\end{align*}
The formal is equivalent to 
$$
H^1(X, M_{L_1} \otimes \wedge^2 M_{L_2} \otimes L_2)=0, 
$$
which holds by Proposition \ref{prop:cohvanA} since $\wedge^2 M_{L_2}$ is a direct summand of $M_{L_2} \otimes M_{L_2}$. 
The latter is equivalent to
$$
H^1(X^{[2]}, M_{N_{2, L_1}} \otimes M_{2, L_1 \otimes L_2} \otimes N_{2, L_2})=0.
$$
We have an exact sequence
$$
0 \longrightarrow S^2 M_{2, L_1} \longrightarrow M_{2, L_1} \otimes H^0(X, L_1) \longrightarrow \wedge^2 H^0(X, L_1) \otimes \mathcal{O}_{X^{[2]}} \longrightarrow \wedge^2 E_{2,L_1} \longrightarrow 0. 
$$
As the last nonzero map is the evaluation map $H^0(X^{[2]}, N_{2, L_1}) \otimes \mathcal{O}_{X^{[2]}} \to N_{2, L_1}$, we get a short exact sequence
$$
0 \longrightarrow S^2 M_{2, L_1} \longrightarrow M_{2, L_1} \otimes H^0(X, L_1) \longrightarrow M_{N_{2, L_1}} \longrightarrow 0. 
$$
Then the desired cohomology vanishing can be deduced from
$$
H^1(X^{[2]}, M_{2, L_1} \otimes M_{2, L_1 \otimes L_2} \otimes N_{2, L_2}) = 0~\text{ and }~ H^2(X^{[2]}, S^2 M_{2, L_1} \otimes M_{2, L_1 \otimes L_2} \otimes N_{2, L_2}) = 0.
$$
The first holds by Proposition \ref{prop:mainvanishing2} $(2)$, and the second holds by Lemma \ref{lem:maintechnical3}.

\end{document}